\documentclass[11pt,a4paper]{article}
\usepackage{epsf,epsfig,amsfonts,amsgen,amsmath,amstext,amsbsy,amsopn,amsthm,amsfonts,amssymb,amscd}
\usepackage{ebezier,eepic}
\usepackage{color}
\usepackage{multirow}
\newtheorem{thm}{Theorem}[section]

\newtheorem{lem}[thm]{Lemma}
\newtheorem{false statement}{False statement}

\theoremstyle{definition}

\newtheorem{claim}{Claim}

\makeatletter \@addtoreset{equation}{section}

\begin{document}
\title{\bf\Large Tur\'{a}n Problems for Vertex-disjoint Cliques in Multi-partite Hypergraphs}
\date{}
\author{ Erica L.L. Liu$^1$, Jian Wang$^2$\\[10pt]
$^{1}$Center for Applied Mathematics\\
Tianjin University\\
Tianjin 300072, P. R. China\\[6pt]
$^{2}$Department of Mathematics\\
Taiyuan University of Technology\\
Taiyuan 030024, P. R. China\\[6pt]
E-mail:  $^1$liulingling@tju.edu.cn, $^2$wangjian01@tyut.edu.cn
}

\maketitle

\begin{abstract}
{For two $s$-uniform hypergraphs $H$ and $F$, the Tur\'{a}n number $ex_s(H,F)$ is the maximum number of edges in an $F$-free subgraph of $H$. Let $s, r, k, n_1, \ldots, n_r$ be integers satisfying $2\leq s\leq r$ and $n_1\leq n_2\leq \cdots\leq n_r$.} De Silva, Heysse and Young determined $ex_2(K_{n_1, \ldots, n_r}, kK_2)$ and De Silva, Heysse, Kapilow, Schenfisch and Young determined  $ex_2(K_{n_1, \ldots, n_r},kK_r)$. In this paper, as a generalization of these results, we consider three Tur\'{a}n-type problems for $k$ disjoint cliques in $r$-partite $s$-uniform hypergraphs. First, we consider a multi-partite version of the Erd\H{o}s matching conjecture and determine $ex_s(K_{n_1, \ldots, n_r}^{(s)},kK_s^{(s)})$ for $n_1\geq s^3k^2+sr$. Then, using a probabilistic argument, we determine $ex_s(K_{n_1, \ldots, n_r}^{(s)},kK_r^{(s)})$ for all $n_1\geq k$. Recently, Alon and Shikhelman determined asymptotically, for all $F$, the generalized Tur\'{a}n number $ex_2(K_n,K_s,F)$, which is the maximum number of copies of $K_s$ in an $F$-free graph on $n$ vertices. Here we determine $ex_2(K_{n_1, \ldots, n_r}, K_s, kK_r)$ with  $n_1\geq k$ and $n_3=\cdots=n_r$.
Utilizing a result on rainbow matchings due to Glebov, Sudakov and Szab\'{o}, we determine  $ex_2(K_{n_1, \ldots, n_r}, K_s, kK_r)$ for all $n_1, \ldots, n_r$ with $n_4\geq r^r(k-1)k^{2r-2}$.


\end{abstract}



\section{Introduction}
An $s$-uniform hypergraph, or simply an $s$-graph, is a hypergraph whose edges have exactly $s$ vertices. For an $s$-graph $H$, let $V(H)$ be the vertex set of $H$ and $E(H)$ the edge set of $H$.  An $s$-graph $H$ is called $F$-free if $H$ does not contain any copy of $F$ as a subgraph. For two $s$-graphs $H$ and $F$, the Tur\'{a}n number $ex_s(H,F)$ is the maximum number of edges of an $F$-free subgraph of $H$. Denote by $K^{(s)}_t$ the complete $s$-graph on $t$ vertices. A copy of $K^{(s)}_t$ in an $s$-graph $H$ is also called a $t$-clique of $H$. Let $kK^{(s)}_t$ denote the $s$-graph consisting of $k$ vertex-disjoint copies of $K^{(s)}_t$. If $t=s$, then $kK^{(s)}_s$ represents a matching of size $k$. Let $n_1, \ldots, n_r$ be integers and $V_1, V_2, \ldots, V_r$ be disjoint vertex sets with $|V_i|=n_i$ for each $i=1,\ldots,r$. A complete $r$-partite $s$-graph on vertex classes $V_1, V_2, \ldots, V_r$, denoted by $K^{(s)}(V_1, V_2, \ldots, V_r)$ or $K_{n_1,n_2, \ldots, n_r}^{(s)}$, is defined to be the $s$-graph whose edge set consists of all the $s$-element subsets $S$ of $V_1\cup V_2\cup\cdots\cup V_r$ such that $|S\cap V_i|\leq 1$ for all $i=1,\ldots,r$. An $s$-graph $H$ is called an $r$-partite $s$-graph on vertex classes $V_1, V_2, \ldots , V_r$ if $H$ is a subgraph of $K^{(s)}(V_1, V_2, \ldots, V_r)$. For $s=2$, we often write $K_t, kK_t, K(V_1, V_2, \ldots, V_r), K_{n_1, n_2,\ldots, n_r}$ and $ex(H,F)$ instead of $K_t^{(2)}, kK_t^{(2)}, K^{(2)}(V_1, V_2, \ldots, V_r),K_{n_1, n_2,\ldots, n_r}^{(2)}$ and $ex_2(H,F)$. Let $[n]$ denote the set $\{1,2,\ldots,n\}$ and $[m,n]$ denote the set $\{m,m+1,\ldots,n\}$ for $m\leq n$.

Tur\'{a}n-type problems were first considered by Mantel \cite{mantel} in 1907, who determined $ex(K_n,K_3)$. In 1941, Tur\'{a}n \cite{turan} showed that the balanced complete $t$-partite graph on $n$ vertices, called the Tur\'{a}n graph and denoted by $T_{n,t}$, is the unique graph that maximises the number of edges among all $K_{t+1}$-free graphs on $n$ vertices.  Since then, Tur\'{a}n numbers of graphs and hypergraphs have been extensively studied. However, even though lots of progress has been made, most of the Tur\'{a}n problems for bipartite graphs and for hypergraphs are still open. Specifically, none of the Tur\'{a}n numbers $ex_s(K_n^{(s)},K_t^{(s)})$ with $t>s>2$ has yet been determined, even asymptotically. We recommend the reader to consult \cite{keevash, sido} for surveys on Tur\'{a}n numbers of graphs and hypergraphs.

Many problems in additive combinatorics are closely related to Tur\'{a}n-type problems in multi-partite graphs and hypergraphs.
Recently, Tur\'{a}n problems in multi-partite graphs have received a lot of attention, see \cite{ben,de2,hanzhao}. The following result, which is attributed to De Silva, Heysse and Young, determines $ex(K_{n_1, \ldots, n_r}, kK_2)$.

\begin{thm}\label{s1}
For $n_1\leq n_2\leq\cdots\leq n_r$ and $k\leq n_1$,
\begin{eqnarray*}
ex(K_{n_1,n_2, \ldots, n_r}, kK_2)=(k-1)(n_2+\cdots+n_r).
\end{eqnarray*}
\end{thm}
Since it seems that their preprint has not been published online, we present a proof of Theorem \ref{s1} in the Appendix for the completeness of the paper. In \cite{de2}, De Silva, Heysse, Kapilow, Schenfisch and Young determined  $ex(K_{n_1, \ldots, n_r},kK_r)$.
\begin{thm}\label{s2}\cite{de2}
For $n_1\leq n_2\leq \cdots\leq n_r$ and $k\leq n_1$,
\begin{eqnarray*}
ex(K_{n_1, \ldots, n_r}, kK_r)=\sum_{1\leq i< j\leq r}n_in_j-n_1n_2+(k-1)n_2.
\end{eqnarray*}
\end{thm}

In this paper, we consider three Tur\'{a}n-type problems for $k$ disjoint cliques in $r$-partite $s$-graphs. Let $n_1,n_2,\ldots,n_r$ be integers. {For any $A\subset [r]$}, denote $\prod_{i\in A}n_i$ by $n_A$. Define
\[
f_{k}^{(s)}(n_2,\ldots,n_r) = (k-1)\sum_{A:A\subset[2,r]\atop |A|=s-1}n_A,
\]
\[
g_{k}^{(s)}(n_1,n_2,\ldots,n_r) = \sum_{A:A\subset [r]\atop|A|=s} n_A -n_{[s]}+(k-1)n_{[2,s]},
\]
and
\[
h_{k}^{(s)}(n_1,n_2,\ldots,n_r) = \sum_{A:A\subset [r]\atop|A|=s,\{1,2\}\not\subset A} n_A +\sum_{A:A\subset [3,r]\atop|A|=s-2}(k-1)n_2 n_A.
\]

\begin{thm}\label{main-4}
For $2\leq s\leq r$, $k\geq1$ and $n_1\leq n_2\leq\cdots\leq n_r$, if $n_1\geq s^3k+sr$ for $s\leq r-2$; $n_1\geq s^3k^2+sr$ for $s=r-1$ and $n_1\geq k$ for $s=r$, then
{
\begin{eqnarray*}
ex_s(K_{n_1,n_2,\ldots,n_r}^{(s)}, kK_s^{(s)})=f_{k}^{(s)}(n_2,\ldots,n_r).
\end{eqnarray*}}
\end{thm}

It should be mentioned that the problem in Theorem \ref{main-4} can be viewed as a multi-partite version of the Erd\H{o}s matching conjecture, which states that
\[
ex_s(K_n^{(s)},kK_s^{(s)})=\max\left\{\binom{ks-1}{s}, \binom{n}{s}-\binom{n-k+1}{s}\right\}
\]
and is still open when $n$ is close to $s(k-1)$, see \cite{bo2,erdos2,frankl0,frankl1,frankl2} for recent progress. The lower bound in Theorem \ref{main-4} follows from the following construction. Let $H_1$ be an $r$-partite $s$-graph on vertex classes $V_1,V_2,\ldots,V_r$ with sizes $n_1,n_2,\ldots,n_r$, respectively. Let $V_1'$ be a $(k-1)$-element subset of $V_1$. An edge $S$ of $K^{(s)}(V_1,V_2,\ldots,V_r)$ forms an edge of $H_1$ if and only if {$S\cap V_1'\neq \emptyset$}. It is easy to see that $H_1$ is $kK_s^{(s)}$-free. Otherwise, if $H_1$ has a matching of size $k$, then we have $|V_1'|\geq k$ since each edge of $H_1$ contains a vertex in $V_1'$.

As our second main result, we use a probabilistic argument to determine $ex_s(K_{n_1, \ldots, n_r}^{(s)},kK_r^{(s)})$.
\begin{thm}\label{main-1}
For $2\leq s\leq r$, $n_1\leq n_2\leq \cdots\leq n_r$ and $k\leq n_1$,
\[
ex_s\left(K_{n_1, \ldots, n_r}^{(s)},kK_r^{(s)}\right)=g_{k}^{(s)}(n_1,n_2,\ldots,n_r).
\]
\end{thm}

The lower bound in Theorem \ref{main-1} follows from the following construction. Let $H_2$ be an $r$-partite $s$-graph on vertex classes $V_1,V_2,\ldots,V_r$ with sizes $n_1,n_2,\ldots,n_r$, respectively.  Let $V_1'$ be an $(n_1-k+1)$-element subset of $V_1$ and let $H_2$ be obtained by deleting all the edges of $K^{(s)}(V_1', V_2, \ldots, V_s)$ from $K^{(s)}(V_1,V_2,\ldots,V_r)$. It is easy to see that $H_2$ is $kK_r^{(s)}$-free. Otherwise, if there are $k$ vertex-disjoint copies of $K_r^{(s)}$ in $H_2$, then we have $|V_1\setminus V_1'|\geq k$ since each copy of $K_r^{(s)}$ in $H_2$ contains a vertex in $V_1\setminus V_1'$.

We also consider the generalized Tur\'{a}n problem in multi-partite graphs. Let $ex(G,T,F)$ denote the maximum number of copies of $T$ in  an $F$-free subgraph of $G$. The first result of this type is due to Zykov \cite{zykov}, who showed that the Tur\'{a}n graph also maximises the number of $s$-cliques in an $n$-vertex $K_{t+1}$-free graph for $s\leq t$. Recently, Alon and Shikhelman \cite{alon} determined $ex(K_n,K_s,F)$ asymptotically for any $F$ with chromatic number $\chi(F)=t+1>s$. Precisely, they proved that
\[
ex(K_n,K_s,F) =k_s(T_{n,t})+o(n^s),
\]
where $k_s(T_{n,t})$ denotes the number of $s$-cliques in the Tur\'{a}n graph $T_{n,t}$.
Later, the error term of this result was further improved by Ma and Qiu \cite{maqiu}.

In this paper, we also study the maximum number of $s$-cliques in a $kK_r$-free subgraph of $K_{n_1, \ldots, n_r}$. By the same probabilistic argument as in the proof of Theorem \ref{main-1}, we obtain the following result.
\begin{thm}\label{main-2}
For $2\leq s\leq r$, $n_1\leq n_2\leq n_3 $ and $k\leq n_1$,
\[
ex(K_{n_1, n_2,\underbrace{n_3,\ldots,n_3}_{r-2}},K_s, kK_r)=h_{k}^{(s)}(n_1,n_2,\underbrace{n_3,\ldots,n_3}_{r-2}).
\]
\end{thm}

Note that for $r=3$, $s\leq 3$ and arbitrary $n_1,n_2,n_3$, the Tur\'{a}n number $ex\left(K_{n_1, n_2, n_3},K_s, kK_3\right)$ is determined by Theorem \ref{main-2}.
Utilizing a result on rainbow matchings due to Glebov, Sudakov and Szab\'{o} \cite{glebov}, we also  determine $ex(K_{n_1, \ldots, n_r},K_s, kK_r)$ for $r\geq 4$ and $n_4$ sufficiently larger than $k$.
\begin{thm}\label{main-3}
For $r\geq 4$, $2\leq s\leq r$, $n_1\leq n_2\leq  \cdots \leq n_r$ and $k\leq n_1$, if $n_4\geq r^r(k-1)k^{2r-2}$, then
\[
ex\left(K_{n_1, \ldots, n_r},K_s, kK_r\right)=h_{k}^{(s)}(n_1,n_2,\ldots,n_r).
\]
\end{thm}

The lower bounds in Theorems \ref{main-2} and \ref{main-3} follow from the same construction as follows. Let $G$ be an $r$-partite graph on $V_1,V_2,\ldots,V_r$, which are of sizes $n_1,n_2,\ldots,n_r$, respectively.  Let $V_1'$ be an $(n_1-k+1)$-element subset of $V_1$. Then $G$ is obtained by deleting all the edges of $K(V_1', V_2)$ from $K(V_1,V_2,\ldots,V_r)$. It is easy to see that $G$ is $kK_r$-free. Otherwise, if there are $k$ vertex-disjoint copies of $K_r$ in $G$, then we have $|V_1\setminus V_1'|\geq k$ since each copy of $K_r$ in $G$ contains a vertex in $V_1\setminus V_1'$.

The rest of the paper is organized as follows. We will prove
Theorem \ref{main-4} in Section 2. In Section 3, we prove Theorem \ref{main-1}. In Section 4, we prove Theorems \ref{main-2} and
\ref{main-3}.

\section{Tur\'{a}n number of $kK_s^{(s)}$ in $r$-partite $s$-graphs}
In this section, we prove Theorem \ref{main-4}. First, we consider the case $s=r$, which is the base case for other results in this paper. Aharoni and Howard \cite{AH17} determined the maximum number of edges in a balanced $r$-partite $r$-graph that is $kK_r^{(r)}$-free. By the same argument, we prove the following result:

\begin{lem}\label{ma}
For any integers $1\leq k\leq n_1\leq n_2\leq\cdots\leq n_r$,
$$ex_r(K^{(r)}_{n_1, \ldots, n_r}, kK_r^{(r)})=(k-1)n_2\cdots n_r.$$
\end{lem}
\begin{proof}
We shall partition the edge set of $K^{(r)}(V_1,\ldots,V_r)$ into $n_2n_3\cdots n_r$ matchings of size $n_1$ each. Let $V_i=\{v_{i,0},v_{i,1},\ldots, v_{i, n_i-1}\}$ for $i=1, 2,\ldots, r$ and
 \[
\Lambda=[0, n_2-1]\times[0, n_3-1]\times\cdots\times [0,n_r-1].
 \]
 For any $(r-1)$-tuple $(x_2, x_3, \ldots, x_r)\in \Lambda$,
define
$$E(x_2, x_3, \ldots, x_r)=\left\{\{v_{1, x}, v_{2,(x+x_2)\bmod {n_2}}, \ldots, v_{r,(x+x_r)\bmod {n_r}}\}:x\in[0,n_1-1]\right\}.$$
It is easy to see that  $E(x_2,x_3, \ldots, x_r)$ is a matching of size $n_1$. Moreover, let
\[
\Omega = \left\{E(x_2, x_3, \ldots, x_r)\colon (x_2, x_3, \ldots, x_r)\in \Lambda \right\}.
\]
We shall show that $\Omega$ forms a partition of the edge set of $K^{(r)}(V_1,\ldots,V_r)$.
On one hand, let  $e=\{v_{1, x_1}, v_{2,x_2}, \ldots, v_{r,x_r}\}$ be an edge in $K^{(r)}(V_1,\ldots,V_r)$ with $v_{i,x_i}\in V_i$ for each $i=1,2,\ldots,r$. Define
\[
y_i :\equiv (x_i-x_1) \bmod n_i
\]
for each $i=2,\ldots,r$. It is easy to check that $e\in E(y_2, y_3, \ldots, y_r)$. Moreover, for each $(x_2, x_3, \ldots, x_r)\in \Lambda$, $E(x_2, x_3, \ldots, x_r)\subset E(K^{(r)}(V_1,\ldots,V_r))$ holds. Thus, we have
\[
E(K^{(r)}(V_1,\ldots,V_r)) = \bigcup_{(x_2, x_3, \ldots, x_r)\in \Lambda} E(x_2, x_3, \ldots, x_r).
\]
On the other hand, for any  two different tuples $(y_2, y_3, \ldots, y_r),(z_2, z_3, \ldots, z_r)\in \Lambda$, we claim that $E(y_2, y_3, \ldots, y_r)\cap E(z_2, z_3, \ldots, z_r)=\emptyset$. Otherwise
if there exists $\{v_{1, x_1}, v_{2,x_2}, \ldots, v_{r,x_r}\}$ $\in E(y_2, y_3, \ldots, y_r)\cap E(z_2, z_3, \ldots, z_r)$, then we have
\[
x_i \equiv (x_1+y_i) \bmod {n_i} \equiv (x_1+z_i)\bmod {n_i}
\]
for all $i=2,\ldots,r$. It follows that $y_i\equiv z_i \bmod {n_i}$. Since $y_i,z_i\in \{0,1,\ldots,n_i-1\}$, we obtain $y_i=z_i$ for all $i=2,\ldots,r$, a contradiction.  Therefore, $\Omega$ forms a partition of the edge set of $K^{(r)}(V_1,\ldots,V_r)$.

Assume that $H\subseteq K^{(r)}(V_1,\ldots,V_r)$  and $e(H)\geq (k-1)n_2\cdots n_r+1$. Then the partition
\[
\left\{E(H)\cap E(x_2,x_3, \ldots, x_r) \colon (x_2, x_3, \ldots, x_r)\in \Lambda \right\}
 \]
of $E(H)$ shows that at least one of the matchings $E(H)\cap E(x_2,x_3, \ldots, x_r)$ has size $k$ or more, a contradiction.

For the lower bound, $K^{(r)}_{k-1,n_2,\ldots, n_r}$ is a $kK^{(r)}_r$-free $r$-graph with $(k-1)n_2\cdots n_r$ edges. Thus, we conclude that $ex_r(K^{(r)}_{n_1, \ldots, n_r},kK^{(r)}_r)=(k-1)n_2\cdots n_r$.
\end{proof}

Let $H$ be an $s$-graph. For $u,v\in V(H)$ and $e\in E(H)$, we define a {\it shifting} operator $S_{uv}$ on $e$ as follows:
\[
S_{uv}(e) =\left\{\begin{array}{ll}
\left(e\setminus \{v\}\right)\cup \{u\},\ \mbox{if} \ v\in e,\ u\notin e\mbox{ and }\left(e\setminus\{v\}\right)\cup \{u\} \notin E(H),\\
e,\ \mbox{otherwise}.
\end{array}\right.
\]
Define $S_{uv}(H)$ be the $s$-graph with vertex set $V(H)$ and edge set $\{S_{uv}(e)\colon e\in E(H)\}$.

It is easy to see that $e(S_{uv}(H))=e(H)$. Let $\nu(H)$ denote the size of a largest matching in $H$.  Frankl \cite{fra-shift} showed that {applying the shifting operator to $H$} does not increase $\nu(H)$. For the completeness we also include a short proof of this.
\begin{lem}\cite{fra-shift}\label{matching}
Let $H$ be an $s$-graph. For any $u,v\in V(H)$,
\[
\nu(S_{uv}(H))\leq \nu(H).
\]
\end{lem}
\begin{proof}
Suppose for contradiction that $\nu(H)=k$ but $\nu(S_{uv}(H))=k+1$. Let $M=\{e_1,e_2,\ldots, e_{k+1}\}$ be a matching of size $k+1$ in $S_{uv}(H)$. Since each edge in $E(S_{uv}(H))\setminus E(H)$ contains $u$, it follows that exactly one of $e_1,e_2,\ldots, e_{k+1}$ is not in $H$. Without loss of generality, we assume that $e_{k+1}\notin E(H)$. Then, $u\in e_{k+1}$, $v\notin e_{k+1}$ and $e_{k+1}'= e_{k+1}\setminus \{u\}\cup \{v\} \in E(H)$. Since $\nu(H)=k$,  it is easy to see that $e_{k+1}' \cap e_i=\{v\}$ for some $i\in [k]$. Since $e_i\in E(H)\cap E(S_{uv}(H))$ and $u\notin e_i$, by the definition of $S_{uv}$ we have $e_i'=e_i\setminus\{v\}\cup \{u\}\in E(H)$. Then, $M\setminus\{e_i,e_{k+1}\}\cup \{e_i',e_{k+1}'\}$ forms a matching of size $k+1$ in $H$, a contradiction.
\end{proof}

Let $H$ be an $r$-partite $s$-graph on vertex classes $V_1,V_2,\ldots ,V_r,$ and
\[
V_i=\{a_{i,1},a_{i,2},\ldots,a_{i,n_i}\}
\]
for $i=1,2,\ldots,r$. Define a partial order $\prec$ on $V = \cup_{i=1}^r V_i$ such that
\[
a_{i,1}\prec a_{i,2}\prec\cdots\prec a_{i,n_i}
\]
for each $i$ and vertices from different parts  are incomparable. For two different edges  $S_1=\{a_1,a_2,\ldots,a_s\}$ and $S_2=\{b_1,b_2,\ldots,b_s\}$ in $K^{(r)}(V_1,\ldots,V_r)$, we define $S_1\prec S_2$ if and only if there exists a permutation $\sigma_1\sigma_2\cdots\sigma_s$ of $[s]$ such that $a_j\prec b_{\sigma_j}$ or $a_j=b_{\sigma_j}$ holds for all $j=1,\ldots,s$.

An $r$-partite $s$-graph $H$ is called a stable $r$-partite $s$-graph if $S_{ab}(H)= H$ holds for all $a,b\in V(H)$ with $a\prec b$. If $H$ is  stable  and $e\in E(H)$, it is easy to see that for any $s$-element vertex subset $S$ with $S\prec e$, we have $S\in E(H)$. Indeed, let $S=\{a_1,a_2,\ldots,a_s\}$ and $e=\{b_1,b_2,\ldots,b_s\}$ . Without loss of generality, we may assume that $a_i\prec b_i$ for each $i=1,\ldots, s_0$ and  $a_i= b_i$ for each $i=s_0+1,\ldots, s$. Since $S_{a_1b_1}(H)= H$ and $e\in E(H)$, it is easy to see that $e_1=e\setminus\{b_1\}\cup \{a_1\}\in E(H)$. Since $S_{a_2b_2}(H)= H$ and $e_1\in E(H)$, it follows that $e_2=e_1\setminus\{b_2\}\cup \{a_2\}\in E(H)$. Repeat the same argument for $i=3,\ldots,s_0$ and we shall obtain that $S\in E(H)$.

To obtain a stable $r$-partite $s$-graph, we can apply the shifting operator to $H$ iteratively. For an intermediate step, let $H^*$ be the current $r$-partite $s$-graph. If $H^*$ is stable, we are done. If $H^*$ is not stable, there exists a pair $(a,b)$ such that $a\prec b$ and $S_{ab}(H^*)\neq H^*$. Then, apply $S_{ab}$ to $H^*$ and we obtain a new $r$-partite $s$-graph. Define
 \[
 g(H^*):= \sum_{e\in E(H^*)} \sum_{i=1}^r\sum_{j:a_{i,j}\in e} j.
 \]
 Since after each step $g(H^*)$ decreases strictly and $g(H)>0$ holds for all the non-empty $r$-partite $s$-graphs $H$, the process will end in finite steps. It should be mentioned that if we apply the shifting operator in different orders, at the end we may arrive at different stable $r$-partite $s$-graphs. For more properties of the shifting operator, we refer the reader to \cite{frankl0}.


For $u,v\in V(H)$, let $L_H(u)$ denote the set of edges in $H$ containing $u$ and $L_H(u,v)$ denote the set of edges in $H$ containing $u$ and $v$. Let $d_H(u)$ and $d_H(u,v)$ denote the cardinality of $L_H(u)$ and $L_H(u,v)$, respectively. For $X\subset V(H)$, let $\Gamma_H(X)$ denote the set of edges in $H$ that intersect $X$. It should be noticed that $\Gamma_H(\{u\})$ is the same as $L_H(u)$. The subscripts will be dropped if there is no confusion. For $S\subset V(H)$, let $H[S]$ denote the  $s$-graph induced by $S$ and $H\setminus S$ the $s$-graph induced by $V(H)\setminus S$.

\begin{lem}\label{la3}
For $3\leq s\leq r-1$, if $n\geq s^3k+sr$ for $s\leq r-2$  and $n\geq s^3k^2+sr$ for $s=r-1$, then
\begin{eqnarray*}
ex_s(K_{\underbrace{n,\ldots,n}_{r}}^{(s)}, kK^{(s)}_s)=(k-1){r-1\choose s-1}n^{s-1}.
\end{eqnarray*}
\end{lem}
\begin{proof}
We prove the lemma by induction on $k$. For $k=1$, the lemma holds trivially. Suppose that the lemma holds for all $k'<k$ and $H$ is a $kK^{(s)}_s$-free subgraph of $K_{\underbrace{n,\ldots,n}_{r}}^{(s)}$ with the maximum number of edges. By Lemma \ref{matching}, we may further assume that $H$ is stable. Let
$T_0=\{a_{1,1}, a_{2,1}, \ldots, a_{r,1}\}$, $\nu(H\setminus T_0)=t$ and $M'=\{e_1, \ldots, e_t\}$ be a largest matching in $H\setminus T_0$. Since $H$ is stable, $H[T_0]$ is not empty. Then, it is easy to see that $t\leq k-2$. Otherwise, for any edge $e\in H[T_0]$, $\{e\}\cup M'$ forms a matching of size $k$ in $H$. Since $H\setminus T_0$ is $(t+1)K^{(s)}_s$-free and $n-1\geq s^3(t+1)+rs$ for $s\leq r-2$ and $n-1\geq s^3(t+1)^2+rs$ for $s= r-1$, by the induction hypothesis, it follows that
\[
e(H\setminus T_0) \leq t{r-1\choose s-1}(n-1)^{s-1}.
\]
If
\[
|\Gamma(T_0)|\leq (k-1){r-1\choose s-1}n^{s-1}-t{r-1\choose s-1}(n-1)^{s-1},
\]
then we conclude that
\[
e(H)= e(H\setminus T_0)+|\Gamma(T_0)| \leq (k-1){r-1\choose s-1}n^{s-1}.
\]
Thus, we are left with the case
\begin{align}\label{lowerLT}
|\Gamma(T_0)|> (k-1){r-1\choose s-1}n^{s-1}-t{r-1\choose s-1}(n-1)^{s-1}.
\end{align}

We will show that inequality (\ref{lowerLT})  either implies the lemma or leads to a contradiction. The proof splits into two cases according to the value of $t$.

\textbf{Case 1.} $t=k-2$. Without loss of generality, assume that  $a_{1,1}$ is the vertex in $T_0$ with the maximum degree within $H$. Since
\[
\sum_{i=1}^r d(a_{i,1}) \geq |\Gamma(T_0)|,
\]
by the inequality \eqref{lowerLT} it follows that
\begin{align*}
d(a_{1,1}) &\geq \frac{1}{r}|\Gamma(T_0)| \\
&> \frac{1}{r}\binom{r-1}{s-1}\left((k-1)n^{s-1}-(k-2)(n-1)^{s-1}\right) \\
&\geq\frac{1}{r}{r-1\choose s-1}n^{s-1}.
\end{align*}

Then, the structure of $H$ can be partly described by the following claim.
\begin{claim}\label{claim-1}
Every edge in $H$ intersects $V_1$.
\end{claim}
\begin{proof}
Suppose to the contrary that there exists an edge in $H$ that does not intersect $V_1$. Since $H$ is stable, there exists an edge in $T_0$ that does not contain $a_{1,1}$. Let
$e_0$ be such an edge.  Let $S$ be the set of vertices covered by the edges in $M'\cup\{e_0\}$, where, as before, $M'$ is a matching of size $k-2$ in $H\setminus T_0$. Clearly, $|S|=(k-1)s$. For each $u\in S$, the number of edges containing $u$ and $a_{1,1}$ is at most ${r-2\choose s-2}n^{s-2}$. Then, there are at most $(k-1)s{r-2\choose s-2}n^{s-2}$ edges in $L(a_{1,1})$ that intersect $S$. It follows that the number of edges in $L(a_{1,1})$ that are {disjoint from} the edges in $M'\cup\{e_0\}$ is at least
\begin{align*}
&d(a_{1,1}) - (k-1)s{r-2\choose s-2}n^{s-2}\\
>& \frac{1}{r}{r-1\choose s-1}n^{s-1}- (k-1)s{r-2\choose s-2}n^{s-2}\\
 =& {r-2\choose s-2} n^{s-2}\left(\frac{r-1}{r(s-1)}n-(k-1)s\right)\\
 >&0,
\end{align*}
where the last inequality follows from the assumption that $n \geq 2s^2k.$ Thus, let $e_0'$ be an edge in $L(a_{1,1})$ that is disjoint from the edges in  $M'\cup\{e_0\}$. Then $M'\cup\{e_0,e_0'\}$ forms a matching of size $k$ in $H$, which contradicts the fact that $H$ is $kK^{(s)}_s$-free. Therefore, the claim holds.
\end{proof}

Define an $r$-partite $r$-graph $H^*$ on vertex classes $V_1, \ldots, V_r$. An $r$-element subset $T$ of $V(H)$ forms an edge of $H^*$ if  $H[T]$ is non-empty and $|T\cap V_i|=1$ for $i=1, \ldots, r$. Since $H$ is $kK^{(s)}_s$-free, it follows that $H^*$ is $kK^{(r)}_r$-free. By Lemma \ref{ma}, we have $e(H^*)\leq (k-1)n^{r-1}$. Now we prove the result by double counting. Let
\[
\Phi = \{(e,T)\colon e\in E(H), T\in E\left(K^{(r)}(V_1, V_2,\ldots, V_r)\right), \mbox{ and }e\subset T \}.
\]
For every $T=\{x_1,x_2,\ldots,x_r\}\in E(H^*)$ with $x_i\in V_i$ for each $i$, since by Claim \ref{claim-1} each edge in $H[T]$ contains $x_1$, it follows that
\[
e(H[T])\leq {r-1\choose s-1}.
\]
Moreover, $H[T]$ is non-empty if and only if $T$ forms an edge in $H^*$. Thus,
\[
|\Phi| \leq (k-1)n^{r-1}{r-1\choose s-1}.
\]
On the other hand, each edge in $H$ appears in $n^{r-s}$ pairs in $\Phi$. Therefore, we have
\[
e(H) = |\Phi|/n^{r-s} \leq (k-1){r-1\choose s-1}n^{s-1}.
\]

\textbf{Case 2.} $t\leq k-3$. Let $X$ be the set of vertices in $T_0$ with  degree greater than $\frac{1}{2r}{r-1\choose s-1}n^{s-1}$ and $Y=T_0\setminus X$.

First, we prove the following claim, which will be used several times.

 \begin{claim}\label{claim-2}
$\nu(H\setminus X)\leq k-1-|X|$.
 \end{claim}
\begin{proof}
 Suppose to the contrary that  $\nu(H\setminus X)\geq k-|X|$. Let $M^*$ be a largest matching in $H\setminus X$. We shall show that $M^*$ can be greedily enlarged to a matching of size $k$ in $H$, which contradicts the fact that $\nu(H)\leq k-1$. Since $\nu(H\setminus X)\geq k-|X|$, it follows that $|X|\geq k-\nu(H\setminus X)= k-|M^*|$. Let $l=k-|M^*|$ and $x_1,x_2,\ldots,x_{l}$ be $l$ vertices in $X$. Set $X_i^+=\{x_{i+1},x_{i+2},\ldots,x_{l}\}$ and $M_0=M^*$.  Note that
\begin{align*}
d(x_1)&\geq \frac{1}{2r}{r-1\choose s-1}n^{s-1}\\
&= \frac{n}{2(s-1)}\cdot\frac{r-1}{r}\cdot{r-2\choose s-2}n^{s-2}\\
&\geq  \frac{n}{2(s-1)}\cdot \frac{2}{3}\cdot{r-2\choose s-2}n^{s-2}\\
&>  sk\cdot{r-2\choose s-2}n^{s-2}\\
&> (|M_0|s+|X_1^+|){r-2\choose s-2}n^{s-2},
\end{align*}
where the second inequality follows from the fact that $r\geq 3$, the third inequality follows form the assumption that $n \geq 3s^2k$ and the last inequality follows from the fact that $k= |M^*|+l>|M_0|+|X_1^+|$. Since there are at most $(|M_0|s+|X_1^+|){r-2\choose s-2}n^{s-2}$ edges in $L(x_1)$ that intersect $(\cup_{e\in M_0} e)\bigcup X_1^+$, we can choose $e_1'$ from $L(x_1)$ such that $M_1=M_0\cup \{e_1'\}$ is a matching of size $|M_0|+1$ and $x_2,x_3,\ldots,x_{l}$ are not used. Now we continue to choose an edge from each of   $L(x_2), \ldots, L(x_{l})$ to enlarge the matching. When dealing with $L(x_i)$,  note that $|X_i^+|=l-i$. Since there are at most $(|M_{i-1}|s+|X_i^+|){r-2\choose s-2}n^{s-2}$ edges in $L(x_i)$ that intersect $(\cup_{e\in M_{i-1}}e)\bigcup X_i^+$ and
\begin{align*}
d(x_i)&\geq \frac{1}{2r}{r-1\choose s-1}n^{s-1}\\
&> sk{r-2\choose s-2}n^{s-2}\\
&> \left(|M_{i-1}|s+|X_i^+|\right){r-2\choose s-2}n^{s-2},
\end{align*}
where the last inequality follows from $k= |M^*|+l>|M_{i-1}|+|X_i^+|$, {therefore} we can choose $e_i'$ from $L(x_i)$ such that $M_{i}=M_{i-1}\cup \{e_i'\}$ is a matching of size $|M_{i-1}|+1$ and $x_{i+1},x_{i+2},\ldots,x_{l}$ are not used. Finally, we end up with $M_{l}$, which is a matching of size $|M^*|+l=k$. It contradicts the fact that $H$ is $kK^{(s)}_s$-free. Thus, we conclude that  $\nu(H\setminus X)\leq k-1-|X|$.
\end{proof}

Then, we show that the sizes of both $X$ and the matching number of $H\setminus X$ can be determined by the matching number of $H\setminus T_0$.

\begin{claim}\label{claim-3}
$|X|= k-1-t$.
\end{claim}
\begin{proof}
 By Claim \ref{claim-2} we have
\[
t=\nu(H\setminus T_0) \leq \nu(H\setminus X) \leq k-1-|X|.
\]
Thus, $|X|\leq k-1-t$. If $|X|\leq k-2-t$, then
\begin{align*}
|\Gamma(T_0)|&\leq |X|{r-1\choose s-1}n^{s-1}+(r-|X|)\cdot\frac{1}{2r}{r-1\choose s-1}n^{s-1}\\
&= |X|{r-1\choose s-1}n^{s-1}\left(1-\frac{1}{2r}\right)+\frac{1}{2}{r-1\choose s-1}n^{s-1}\\
&\leq\left((k-2-t)\left(1-\frac{1}{2r}\right)+\frac{1}{2}\right){r-1\choose s-1}n^{s-1}\\
&<(k-1-t){r-1\choose s-1}n^{s-1},
\end{align*}
which contradicts the inequality \eqref{lowerLT}. Thus, the claim holds.
\end{proof}
\begin{claim}\label{claim-4}
$\nu(H\setminus X) =\nu(H\setminus T_0)=t$.
\end{claim}
\begin{proof}
By Claims \ref{claim-2} and \ref{claim-3}, we have
\[
\nu(H\setminus X)\leq k-1-|X| = t =\nu(H\setminus T_0).
\]
Moreover, since $H\setminus T_0$ is a subgraph of $H\setminus X$, it follows that
$\nu(H\setminus X)\geq \nu(H\setminus T_0)$.  Thus, the claim holds.
\end{proof}

We also claim that $Y$ cannot be an empty set. Otherwise, by Claim \ref{claim-3} we have
\[
|\Gamma(T_0)| =|\Gamma(X)| \leq |X|\binom{r-1}{s-1} n^{s-1}=(k-1-t)\binom{r-1}{s-1} n^{s-1},
\]
which contradicts the inequality \eqref{lowerLT}.



By Claim \ref{claim-4}, we have that all edges in $L_{H\setminus X}(y)$ {intersect} $\cup_{e\in M'} e$ for each $y\in Y$. Otherwise, if there exists an edge $e_0$ in $L_{H\setminus X}(y)$ that is disjoint from $\cup_{e\in M'} e$ for some $y\in Y$, then $M'\cup \{e_0\}$ forms a matching of size $t+1$ in $H\setminus X$, a contradiction.   Then, we can obtain an upper bound on $|\Gamma_{H\setminus X}(Y)|$ by the following argument. For $e_i\in M'$, define a bipartite graph $G_i$ on vertex classes $Y$ and $e_i$, where $e_i$ is viewed as one of the sides of $G_i$. For $u\in e_i$ and $v\in Y$, $\{u,v\}$ is an edge of $G_i$ if $d_{H\setminus X}(u, v)>(t+1)s{r-3\choose s-3}n^{s-3}$. If there is an $i$ such that $\nu(G_i)\geq 2$, let $\{u_p,v_p\}$ and $\{u_q,v_q\}$ be two disjoint edges of $G_i$ {with $u_p,u_q\in e_i$ and $v_p,v_q\in Y$}. Since there are at most $ts{r-3\choose s-3}n^{s-3}$ edges in $L_{H\setminus X}(u_p,v_p)$ that intersect  $(\cup_{e\in M'} e)\cup \{v_q\}\setminus \{u_p\}$,  we can find an edge $f_p$ in $L_{H\setminus X}(u_p,v_p)$ that is disjoint from $(\cup_{e\in M'} e)\cup \{v_q\}\setminus \{u_p\}$. Similarly, there are at most $(t+1)s{r-3\choose s-3}n^{s-3}$ edges in $L_{H\setminus X}(u_q,v_q)$ that intersect  $(\cup_{e\in M'} e)\cup \{f_p\}\setminus \{u_q\}$. Thus, we can find an edge $f_q$ in $L_{H\setminus X}(u_q,v_q)$ that is disjoint from {$(\cup_{e\in M'} e)\cup \{f_p\}\setminus \{u_q\}$}. Now $(M'\setminus\{e_i\})\cup \{f_p,f_q\}$ forms a matching of size $t+1$ in $H\setminus X$, which contradicts with Claim \ref{claim-4}. Thus, we conclude that each $G_i$ has matching number at most one.

Let $e_i\in M'$ and
\[
\Gamma_{H\setminus X}(e_i,Y) =\{e\in E(H\setminus X)\colon e\cap e_i\neq \emptyset \mbox{ and } e\cap Y\neq \emptyset\}.
\]
The rest of the proof is divided into two subcases according to the size of $|Y|$.

\textbf{Case 2.1.} $|Y|\geq s$. Since $\nu(G_i)\leq 1$, by Lemma \ref{ma}, there are at most $|Y|$ edges in $G_i$.
Then,
\begin{align*}
|\Gamma_{H\setminus X}(e_i,Y)| &\leq e(G_i){r-2\choose s-2}n^{s-2} +\left(|Y||e_i|-e(G_i)\right)(t+1)s{r-3\choose s-3}n^{s-3}\\[5pt]
&= e(G_i){r-3\choose s-3}n^{s-3}\left(\frac{r-2}{s-2}n-(t+1)s\right)+|Y||e_i|(t+1)s{r-3\choose s-3}n^{s-3}\\[5pt]
&\leq |Y|{r-3\choose s-3}n^{s-3}\left(\frac{r-2}{s-2}n-(t+1)s\right)+|Y|s^2(t+1){r-3\choose s-3}n^{s-3},\\
&= (r-|X|){r-3\choose s-3}n^{s-3}\left(\frac{r-2}{s-2}n+(t+1)s(s-1)\right),
\end{align*}
where the second inequality follows from the assumption that $n>sk$ and $e(G_i)\leq |Y|$.

By Claim \ref{claim-4}, for any $y\in Y$ all edges in $L_{H\setminus X}(y)$ intersect $\cup_{e\in M'} e$.  It follows that
\begin{align*}
|\Gamma_{H\setminus X}(Y)| &\leq \sum_{i=1}^t |\Gamma_{H\setminus X}(e_i,Y)|\leq t(r-|X|){r-3\choose s-3}n^{s-3}\left(\frac{r-2}{s-2}n+(t+1)s(s-1)\right).
\end{align*}
Since
\[
|\Gamma(X)| \leq \sum_{i=1}^{|X|} d(x_i) \leq |X| {r-1\choose s-1}n^{s-1},
\]
{therefore}
\begin{align}\label{upperLT}
|\Gamma(T_0)| &=|\Gamma(X)|+|\Gamma_{H\setminus X}(Y)|\nonumber\\[5pt]
&\leq |X| {r-1\choose s-1}n^{s-1} + t(r-|X|){r-3\choose s-3}n^{s-3}\left(\frac{r-2}{s-2}n+(t+1)s(s-1)\right).
\end{align}
By combining the inequalities \eqref{lowerLT} and \eqref{upperLT} and using the fact that $|X| = k-1-t$, we arrive at
\begin{align*}
t{r-1\choose s-1}\left(n^{s-1}-(n-1)^{s-1}\right)\leq t(r-|X|){r-3\choose s-3}n^{s-3}\left(\frac{r-2}{s-2}n+(t+1)s(s-1)\right).
\end{align*}

Since $|X|\geq 2$ (because $|X| = k-1-t$ and in Case 2 we assume that $t \leq k-3$), we have
\begin{align}\label{middineq}
\frac{(r-1)(r-2)}{(s-1)(s-2)}\left(n^{s-1}-(n-1)^{s-1}\right)\leq (r-2)n^{s-3}\left(\frac{r-2}{s-2}n+(t+1)s(s-1)\right).
\end{align}
By Taylor's Theorem with Lagrange remainder, it can be deduced that
\begin{align}\label{taylor1}
n^{s-1}-(n-1)^{s-1}\geq (s-1) n^{s-2} -\frac{(s-1)(s-2)}{2}n^{s-3}.
\end{align}
By combining the inequalities \eqref{middineq} and \eqref{taylor1}, we obtain that
\begin{align}\label{finalineq}
\frac{r-1}{s-2} n^{s-2}-\frac{r-1}{2}n^{s-3}\leq n^{s-3}\left(\frac{r-2}{s-2}n+(t+1)s(s-1)\right).
\end{align}
Since $t\leq k-3$, by simplifying the inequality \eqref{finalineq} we arrive at
{\[
 n\leq s(s-1)(s-2)(k-2)+\frac{(r-1)(s-2)}{2}<s^3k+sr,
\]}
which contradicts the fact that $n\geq s^3k+sr$.

\textbf{Case 2.2.} $|Y|\leq s-1$.

For each $i=1,2,\ldots,t$, since $\nu(G_i)\leq 1$, by Lemma \ref{ma} we have $e(G_i)\leq s$.
Then
\begin{align*}
|\Gamma_{H\setminus X}(e_i,Y)| &\leq e(G_i){r-2\choose s-2}n^{s-2} +\left(|Y||e_i|-e(G_i)\right)(t+1)s{r-3\choose s-3}n^{s-3}\\[5pt]
&= e(G_i){r-3\choose s-3}n^{s-3}\left(\frac{r-2}{s-2}n-(t+1)s\right)+|Y||e_i|(t+1)s{r-3\choose s-3}n^{s-3}\\[5pt]
&\leq s{r-3\choose s-3}n^{s-3}\left(\frac{r-2}{s-2}n-(t+1)s\right)+|Y|s^2(t+1){r-3\choose s-3}n^{s-3},\\
&= s{r-3\choose s-3}n^{s-3}\left(\frac{r-2}{s-2}n+(t+1)s(|Y|-1)\right),
\end{align*}
where the second inequality follows from the assumption that $n>sk$ and $e(G_i)\leq s$. Thus,
\begin{align}\label{hxyineq}
|\Gamma_{H\setminus X}(Y)| &\leq \sum_{i=1}^t |\Gamma_{H\setminus X}(e_i,Y)|\leq st{r-3\choose s-3}n^{s-3}\left(\frac{r-2}{s-2}n+(t+1)s(|Y|-1)\right).
\end{align}
Since
\[
|\Gamma(X)| \leq \sum_{i=1}^{|X|} d(x_i) \leq |X| {r-1\choose s-1}n^{s-1},
\]
therefore
\begin{align}\label{upperLT2}
|\Gamma(T_0)| &=|\Gamma(X)|+|\Gamma_{H\setminus X}(Y)|\nonumber\\[5pt]
&\leq |X|{r-1\choose s-1}n^{s-1}+st{r-3\choose s-3}n^{s-3}\left(\frac{r-2}{s-2}n+(t+1)s(|Y|-1)\right).
\end{align}
By combining the inequalities \eqref{lowerLT} and \eqref{upperLT2} and using Claim 3, we arrive at
\begin{align}\label{middineq2}
t{r-1\choose s-1}\left(n^{s-1}-(n-1)^{s-1}\right)\leq st{r-3\choose s-3}n^{s-3}\left(\frac{r-2}{s-2}n+(t+1)s(|Y|-1)\right).
\end{align}
Then by combining the inequalities \eqref{taylor1} and \eqref{middineq2} we obtain that
\begin{align}\label{finalineq2}
\frac{(r-1)(r-2)}{s-2}n^{s-2}-\frac{(r-1)(r-2)}{2}n^{s-3}\leq sn^{s-3}\left(\frac{r-2}{s-2}n+(t+1)s(|Y|-1)\right).
\end{align}
By simplifying, we arrive at
\[
 (r-1-s)n\leq \frac{(r-1)(s-2)}{2} + \frac{s^2(s-2)(t+1)(|Y|-1)}{r-2}.
\]
Since $|Y|=r-|X|\leq  r-2$ and $t+1\leq k$, it follow that
\begin{align*}
 (r-1-s)n&\leq \frac{(r-1)(s-2)}{2} + s^2(s-2)k.
\end{align*}
Since $n\geq s^3k+sr$ when $s\leq r-2$, it leads to a contradiction for $s\leq r-2$.

For $r=s+1$, we shall give a slightly better upper bound on $|L(X)|$ as follows. Let $X=\{x_1,\ldots,x_{k-1-t}\}$, $X_{0} = \emptyset$ and $X_i =\{x_1,\ldots,x_i\}$ for $i=1,2,\ldots,k-1-t$. Note that
\[
|\Gamma(X)| = \sum_{i=1}^{|X|} |L_{H\setminus X_{i-1}}(x_i)|.
\]
Now, it is easy to see that
\[
L_{H\setminus X_{1}}(x_2)\leq  {s-1\choose s-1}n^{s-1}+(n-1){s-1\choose s-2}n^{s-2} = sn^{s-1} - (s-1)n^{s-2}.
\]
For $i\neq 2$, we use the trivial inequality  $|L_{H\setminus X_{i-1}}(x_i)|\leq sn^{s-1}$.
Since $|X|\geq 2$ ({because $|X| = k-1-t$ and in Case 2 we assume that $t \leq k-3$)}, it follows that
\[
|\Gamma(X)| = \sum_{i=1}^{|X|} |L_{H\setminus X_{i-1}}(x_i)|\leq |X|sn^{s-1} - (s-1)n^{s-2}.
\]
Then, by the inequality \eqref{hxyineq}, we obtain an upper bound on $\Gamma(T_0)$ as follows:
\begin{align}\label{upperLT3}
|\Gamma(T_0)| &=|\Gamma(X)|+|\Gamma_{H\setminus X}(Y)|\nonumber\\[5pt]
&\leq |X|sn^{s-1} - (s-1)n^{s-2}+st{r-3\choose s-3}n^{s-3}\left(\frac{r-2}{s-2}n+(t+1)s(|Y|-1)\right)\nonumber\\[5pt]
&= |X|sn^{s-1} - (s-1)n^{s-2}+st(s-2)n^{s-3}\left(\frac{s-1}{s-2}n+(t+1)s(|Y|-1)\right).
\end{align}
By combining the inequalities \eqref{lowerLT} and \eqref{upperLT3} and using Claim 3,  we have
\begin{align}\label{middineq3}
ts\left(n^{s-1}-(n-1)^{s-1}\right)\leq st(s-2)n^{s-3}\left(\frac{s-1}{s-2}n+(t+1)s(|Y|-1)\right) -(s-1)n^{s-2}.
\end{align}
By simplifying, we obtain that
\begin{align}\label{middineq4}
ts\left(n^{s-1}-(n-1)^{s-1}\right)\leq (st(s-1)-s+1) n^{s-2}+s^2(s-2)t(t+1)(|Y|-1)n^{s-3}.
\end{align}
By combining the inequalities \eqref{taylor1} and \eqref{middineq4}, we arrive at
\[
(s-1)n \leq \frac{s(s-1)(s-2)t}{2} +s^2(s-2)t(t+1)(|Y|-1).
\]
Since $|Y|= s+1-|X|\leq s-1$ and $t+1\leq k$, it follows that
\[
n\leq \frac{s(s-2)}{2}t+s^2 t(t+1)(|Y|-1)\cdot\frac{s-2}{s-1} \leq \frac{s^2}{2}k+ s^2(s-2)k^2\leq s^3k^2,
\]
which contradicts the fact that $n\geq s^3k^2+sr$ for $s=r-1$.

Thus, we complete the proof of Lemma \ref{la3}.
\end{proof}
In the following proof of Theorem \ref{main-4}, we shall use Theorem \ref{s1} and Lemma \ref{la3} as base cases.

{\noindent \it Proof of Theorem \ref{main-4}.} Notice that Lemma \ref{ma} implies the theorem for $s=r$. So we are left with the case $s\leq r-1$. We prove by induction on $(s,\sum_{i=2}^r (n_i-n_1))$.
 The base case of $s=2$ is verified for all $r$ and $n_1\leq n_2\leq \cdots\leq n_r$ by Theorem \ref{s1}. For every $s\geq 3$, the base case of $\sum_{i=2}^r (n_i-n_1)=0$ is verified for all $r$ by  Lemma \ref{la3}.  Suppose now that $\sum_{i=2}^r (n_i-n_1)>0$. Assume that for all $r$, the theorem holds for all pairs $(s',\sum_{i=2}^r (n_i'-n_1'))$ such that $s'<s$ or $s'=s$ together with $\sum_{i=2}^r (n_i'-n_1')<\sum_{i=2}^r (n_i-n_1)$.
  There exists an $i\in [2,r]$ such that $n_i>n_{i-1}$. Without loss of generality, assume that $i=r$. Let $H$ be a $kK^{(s)}_s$-free subgraph of $K^{(s)}_{n_1,\ldots,n_r}$. By Lemma \ref{matching} we may assume that $H$ is stable. Let $V_r$ be the vertex set with cardinality $n_r$ and
\[
V_r=\{a_{r,1},a_{r,2},\ldots,a_{r,n_r}\}.
\]
Let $H'=H\setminus\{a_{r,n_r}\}$ and
\[
H(a_{r,n_r}) = \{S\subset V\colon S\cup\{a_{r,n_r}\}\in E(H)\}.
\]
Clearly, $H(a_{r,n_r})$ is an $(r-1)$-partite $(s-1)$-graph with parts of sizes $n_1,n_2,\ldots,n_{r-1}$. We claim that $\nu(H(a_{r,n_r}))\leq k-1$. Otherwise, suppose $M=\{e_1,e_2,\ldots,e_k\}$ is a matching of size $k$ in $H(a_{r,n_r})$. Since $H$ is stable and $n_r>k$,
$\{e_1\cup\{a_{r,1}\},e_2\cup\{a_{r,2}\},\ldots,e_k\cup\{a_{r,k}\}\}$ forms a matching of size $k$, which contradicts the fact that $H$ is $kK^{(s)}_s$-free. Since $H'$ is $kK^{(s)}_s$-free, by the induction hypothesis on $\sum_{i=2}^r (n_i-n_1)$, we have
{
\[
e(H')\leq f_{k}^{(s)}(n_2\ldots,n_{r-1},n_{r}-1).
\]
}
Since $H(a_{r,n_r})$ is a $kK^{(s-1)}_{s-1}$-free $(r-1)$-partite $(s-1)$-graph, $n_1\geq s^3k+sr\geq (s-1)^3k+(s-1)(r-1)$ for $(s-1)\leq (r-1)-2$ and $n_1\geq s^3k^2+sr\geq (s-1)^3k^2+(s-1)(r-1)$ for $(s-1)=(r-1)-1$, by the induction hypothesis on $s$, we have
{
\[
e(H(a_{r,n_r})) \leq f_{k}^{(s-1)}(n_2\ldots,n_{r-1}).
\]
}
Thus,
\begin{align*}
e(H)&= e(H')+e(H(a_{r,n_r})) \\
&\leq f^{(s)}_k(n_2,\ldots,n_{r-1},n_r - 1) + f^{(s-1)}_k(n_2,\ldots,n_{r-1})\\
&=f_{k}^{(s)}(n_2\ldots,n_{r-1},n_{r}),
\end{align*}
which completes the proof.
\qed

\section{Tur\'{a}n number of $kK^{(s)}_r$ in $r$-partite $s$-graphs}

In this section, we generalize the result of \cite{de2} to $s$-graphs by using a probabilistic argument. The following lemma will be useful for us.
\begin{lem}\label{op}
Assume that $b>0$, $w_1\geq w_2\geq \cdots \geq w_N>0$ and let $(P)$ be a linear programming model as follows:
\begin{align*}
\max\quad &z=\sum_{i=1}^N x_i\\
s.t.\quad &\sum_{i=1}^N w_i^{-1} x_i\leq b, \\
 &0\leq x_i\leq w_i, \quad i=1,2,\ldots,N.
\end{align*}
Let $M$ be the integral part of $b$ and $a= w_{M+1}(b - M)$.
Then $\sum_{i=1}^{M} w_i +a$ is the optimal value of $(P)$.
\end{lem}

\begin{proof} Suppose to the contrary that there exists a feasible solution $y=(y_1,y_2,\ldots,y_N)$ to $(P)$ such that
\[
\sum_{i=1}^N y_i > \sum_{i=1}^{M} w_i +a.
\]
Since $y$ is a feasible solution, it follows that
\[
\sum\limits_{i=1}^N w_i^{-1} y_i \leq b=M+w_{M+1}^{-1}a=\sum\limits_{i=1}^{M} w_i^{-1} w_i+w_{M+1}^{-1}a.
\]
Then, since $w_i\geq w_j$ for any $i<j$, we have
\[
\sum\limits_{i=1}^{M} w_i^{-1} (w_i-y_i)\geq \sum\limits_{i=M+1}^N w_i^{-1} y_i -w_{M+1}^{-1}a \geq  w_{M+1}^{-1}\left(\sum\limits_{i=M+1}^N y_i -a \right)>w_{M+1}^{-1}\sum\limits_{i=1}^{M} (w_i-y_i) .
\]
On the other hand, since
\[
\sum_{i=1}^{M} w_i^{-1} (w_i-y_i)\leq w_{M}^{-1}\sum_{i=1}^{M} (w_i-y_i),
\]
we arrived at $w_{M}^{-1}>w_{M+1}^{-1}$, a contradiction. Thus, the lemma holds.
\end{proof}

Let $H$ be an $r$-partite $s$-graph on vertex classes $V_1,V_2,\ldots,V_r$.
For any $A\subset [r]$, we shall write $\cup_{i\in A} V_i$ as $V_A$ for short. Denote by $E(V_A)$ the edge set of the induced subgraph $H[V_A]$ and $e(V_A)$ the cardinality of $E(V_A)$.

{\noindent \it Proof of Theorem \ref{main-1}.}
Suppose $H\subseteq K^{(s)}_{n_1, \ldots, n_r}$ does not contain any copy of $kK^{(s)}_r$. Choose an $r$-tuple $(x_1, x_2, \ldots, x_r)$ from $V_1\times V_2\times\cdots\times V_r$ uniformly at random. Let $T=\{x_1, x_2, \ldots, x_r\}$ and  $X(T)$ be the number of edges in $H[T]$. Then
\begin{align}\label{eq1}
\mathbb{E}(X(T))=\sum_{S\in E(H)}Pr(S\subset T)=\sum_{A:A\subset [r]\atop|A|=s}\sum_{S\in E(V_A)}\frac{1}{n_A}=\sum_{A:A\subset [r]\atop|A|=s}\frac{e(V_A)}{n_A}.
\end{align}

On the other hand, let $m$ be the number of copies of $K^{(s)}_r$ in $H$. Define an $r$-partite $r$-graph $H^*$ on the same vertex classes $V_1, V_2, \ldots, V_r$.  An $r$-element set $S$ forms an edge  in $H^*$ if and only if $H[S]$ is a copy of $K^{(s)}_r$. Since $H$ is $kK_r^{(s)}$-free, it follows that the matching number of $H^*$ is at most $k-1$. Moreover, the number of edges in $H^*$ is exactly $m$.  Then, by  Lemma \ref{ma}, we have $m\leq(k-1)n_2\cdots n_r$. Let $A_T$ be the event that $H[T]$ is a copy of $K_r^{(s)}$. Clearly, we have
\[
Pr(A_T)= \frac{m}{n_1n_1\cdots n_r} \leq \frac{k-1}{n_1}.
\]
Thus,
\begin{align}\label{eq2}
\mathbb{E}(X(T))&= \mathbb{E}(X(T)|A_T) Pr(A_T) +\mathbb{E}(X(T)|\overline{A_T}) Pr(\overline{A_T})\nonumber\\[5pt]
&\leq{r\choose s} Pr(A_T)+\left({r\choose s}-1\right)(1- Pr(A_T))\nonumber\\
&={r\choose s}-1+Pr(A_T)\nonumber\\
&\leq{r\choose s}-1+\frac{k-1}{n_1}.
\end{align}
Putting (\ref{eq1}) and (\ref{eq2}) together, we obtain that
\begin{eqnarray}\label{eq-op1}
\sum_{A:A\subset [r]\atop|A|=s}e(V_A)\frac{1}{n_A}\leq{r\choose s}-1+\frac{k-1}{n_1}.
\end{eqnarray}
We consider the linear programming model (P1) as follows:
\begin{align*}
\max\quad &z=\sum_{A:A\subset [r]\atop|A|=s} x_A\\
s.t.\quad&\sum\limits_{A:A\subset [r]\atop|A|=s} n_A^{-1} x_A\leq {r\choose s}-1+\frac{k-1}{n_1},\\
&0\leq x_A\leq n_A, \ A\in \binom{[r]}{s}.
\end{align*}

Applying Lemma \ref{op} by setting $N={r\choose s}$, $b={r\choose s}-1+\frac{k-1}{n_1}$ and $w_i$ be the $i$-th largest value in $\{n_A\colon A\in \binom{[r]}{s}\}$ for each $i\in 1,2,\ldots,\binom{r}{s}$ in (P), we have $M=\lfloor b \rfloor = {r\choose s}-1$. Since $n_{[s]}\leq n_A$ for all $A\in \binom{[r]}{s}$, it follows that
\[
a= w_{M+1}(b - M) =n_{[s]} \cdot \frac{k-1}{n_1}=(k-1)n_{[2,s]}.
\]
Thus, the optimal value of (P1) is
\begin{align*}
\sum_{i=1}^{M} w_i +a &= \sum_{A:A\subset [r]\atop|A|=s, A\neq [s]} n_A + (k-1)n_{[2,s]}= g_{k}^{(s)}(n_1,n_2,\ldots,n_r).
\end{align*}

Let $y$ be a vector indexed by the $s$-element subset $A$ of $[r]$ with $y_A=e(V_A)$. Since $e(V_A)\leq n_A$ and the inequality \eqref{eq-op1} holds, it follows that $y$ is a feasible solution to (P1). Therefore, we have
\begin{eqnarray*}
e(H)=\sum_{A:A\subset [r]\atop|A|=s}e(V_A)=\sum_{A:A\subset [r]\atop|A|=s}y_A\leq  g_{k}^{(s)}(n_1,n_2,\ldots,n_r).
\end{eqnarray*}
Thus, the theorem follows.
\qed

\section{The number of $s$-cliques in $r$-partite graphs}
In this section, we first determine $ex(K_{n_1, \ldots, n_r},K_s, kK_r)$ for the case $n_1\leq n_2\leq n_3 =n_4=\cdots = n_r$. Then, by utilizing a result on rainbow matchings, we  determine  $ex(K_{n_1, \ldots, n_r}, K_s, kK_r)$ for all $n_1, \ldots, n_r$ with $n_4\geq r^r(k-1)k^{2r-2}$.

For an $r$-partite graph $G$ on vertex classes $V_1,V_2,\ldots, V_r$, we use $K_s(G)$ to denote the family of $s$-element subsets of $V(G)$ that form $s$-cliques in $G$ and for $u\in V(G)$ we use $K_s(u,G)$ to denote the family of $s$-element subsets in $K_s(G)$ that contain $u$. For any $A\subset [r]$, we also use $K_s(V_A)$ to denote $K_s(G[V_A])$. Let $k_s(G)$, $k_s(u,G)$ and $k_s(V_A)$ be the cardinalities of $K_s(G)$, $K_s(u,G)$ and $K_s(V_A)$, respectively.

{\noindent \it Proof of Theorem \ref{main-2}.} Let $V_1,V_2,\ldots, V_r$ be the  vertex classes such that $|V_i|=n_i$ for each $i=1,2,\ldots,r$ and $n_4=\ldots=n_r=n_3$.
Suppose $G\subseteq K(V_1,V_2,\ldots,V_r)$ does not contain any copy of $kK_r$. Choose an $r$-tuple $(x_1, x_2, \ldots, x_r)$ from $V_1\times V_2\times\cdots\times V_r$ uniformly at random.
Let $T=\{x_1, x_2, \ldots, x_r\}$ and $X(T)$ be the number of copies of $K_s$ in $G[T]$. Then,
\begin{align}\label{eq3}
\mathbb{E}(X(T))=\sum_{S\in K_s(G)}Pr(S\subset T)=\sum_{A:A\subset [r]\atop|A|=s}\sum_{S\in K_s(V_A)}\frac{1}{n_A}
=\sum_{A:A\subset [r]\atop|A|=s}\frac{k_s(V_A)}{n_A}.
\end{align}

On the other hand, let $m$ be the number of copies of $K_r$ in $G$.  By a similar argument as in the proof of Theorem \ref{main-1}, we have $m\leq(k-1)n_2n_3^{r-2}$. If $s=r$, then the theorem holds already  (because $h^{(r)}_k(n_1,n_2,\ldots,n_r) = (k-1)n_2\ldots n_r$),  so we are left with the case $s\leq r-1$.  Let $A_T$ be the event that $H[T]$ is a copy of $K_r$. Clearly, we have
\[
Pr(A_T)\leq  \frac{k-1}{n_1}.
\]
Since there are $\binom{r}{s}$ $s$-cliques in $K_r$ and at most ${r\choose s}-{r-2\choose s-2}$ $s$-cliques in a graph on $r$ vertices that is not a complete graph, it follows that
\begin{align}\label{eq4}
\mathbb{E}(X(T))&= \mathbb{E}(X(T)|A_T) Pr(A_T) +\mathbb{E}(X(T)|\overline{A_T}) Pr(\overline{A_T})\nonumber\\[5pt]
&\leq{r\choose s}Pr(A_T)+\left({r\choose s}-{r-2\choose s-2}\right)(1-Pr(A_T))\nonumber\\
&={r\choose s}-{r-2\choose s-2}+{r-2\choose s-2}Pr(A_T)\nonumber\\
&\leq{r\choose s}-{r-2\choose s-2}+\frac{k-1}{n_1}{r-2\choose s-2}.
\end{align}

Combining (\ref{eq3}) and (\ref{eq4}), we have
\begin{align}\label{ineq-op2}
\sum_{A:A\subset [r]\atop|A|=s}k_s(V_A)\frac{1}{n_A}\leq {r\choose s}-{r-2\choose s-2}+\frac{k-1}{n_1}{r-2\choose s-2}.
\end{align}

We consider the linear programming model (P2) as follows:
\begin{align*}
\max\quad &z=\sum_{A:A\subset [r]\atop|A|=s} x_A\\
s.t.\quad&\sum\limits_{A:A\subset [r]\atop|A|=s} n_A^{-1} x_A\leq {r\choose s}-{r-2\choose s-2}+\frac{k-1}{n_1}{r-2\choose s-2},\\
&0\leq x_A\leq n_A, \ A\in \binom{[r]}{s}.
\end{align*}

Apply Lemma \ref{op} by setting $N={r\choose s}$, $b={r\choose s}-{r-2\choose s-2}+\frac{k-1}{n_1}{r-2\choose s-2}$ and $w_i$ be the $i$-th largest value in $\{n_A\colon A\in \binom{[r]}{s}\}$ for each $i\in 1,2,\ldots,\binom{r}{s}$ in (P). Note that $n_A=n_3^s$ for $A\in {[3,r]\choose s}$,
$n_A=n_2n_3^{s-1}$ for $A\in {[2,r]\choose s}$ and $2\in A$, $n_A=n_1n_3^{s-1}$ for $A\in {[r]\setminus \{2\}\choose s}$ and $1\in A$, $n_A=n_1n_2n_3^{s-2}$ for $A\in {[r]\choose s}$ and $\{1,2\}\subset A$. Since $n_3^s\geq n_2n_3^{s-1}\geq n_1n_3^{s-1} \geq n_1n_2n_3^{s-2}$ and
\[
\binom{r-2}{s}+ \binom{r-2}{s-1}+\binom{r-2}{s-1} = \binom{r-1}{s}+\binom{r-2}{s-1} = \binom{r}{s} -\binom{r-2}{s-2},
\]
it follows that $w_{i} = n_1n_2n_3^{s-2}$ for $i\geq \binom{r}{s} -\binom{r-2}{s-2}+1$. Since $M=\lfloor b\rfloor \geq \binom{r}{s} -\binom{r-2}{s-2}$, we have $w_{M+1} = n_1n_2n_3^{s-2}$.
Thus, the optimal value of (P2) is
\begin{align*}
\sum_{i=1}^{M} w_i +a =&\sum_{i=1}^{\binom{r}{s} -\binom{r-2}{s-2}} w_i+ \sum_{i=\binom{r}{s} -\binom{r-2}{s-2}+1}^Mw_i+w_{M+1}(b - M)\\[5pt]
=&\sum_{A:A\subset [r]\atop|A|=s,\{1,2\}\not\subset A} n_A +\sum_{i=\binom{r}{s} -\binom{r-2}{s-2}+1}^Mw_{M+1}+(b - M)w_{M+1}\\[5pt]
=& \sum_{A:A\subset [r]\atop|A|=s,\{1,2\}\not\subset A} n_A  +\left(M-\binom{r}{s} -\binom{r-2}{s-2}+b - M\right)n_1n_2n_3^{s-2}\\[5pt]
=& \sum_{A:A\subset [r]\atop|A|=s,\{1,2\}\not\subset A} n_A  +\frac{k-1}{n_1}{r-2\choose s-2}n_1n_2n_3^{s-2}\\
=&\sum_{A:A\subset [r]\atop|A|=s,\{1,2\}\not\subset A} n_A +\sum_{A:A\subset [3,r]\atop|A|=s-2}(k-1)n_2 n_A\\
=&h_{k}^{(s)}(n_1,n_2,\underbrace{n_3,\ldots,n_3}_{r-2}).
\end{align*}

Let $y$ be a vector indexed by $s$-element subset $A$ of $[r]$ with $y_A=k_s(V_A)$. Since $k_s(V_A)\leq n_A$ and the inequality \eqref{ineq-op2} holds, it follows that $y$ is a feasible solution to (P2).  Therefore, we obtain that
\[
k_s(G)= \sum_{A:A\subset [r]\atop|A|=s} k_s(V_A) \leq  h_{k}^{(s)}(n_1,n_2,\underbrace{n_3,\ldots,n_3}_{r-2}).
\]
Thus, the theorem holds.
\qed
\\

Let $f,k\geq 1$ be integers. A $k$-matching is a matching of size $k$. Given a coloring $c: E(G) \rightarrow[f]$ of the edges of an $r$-graph $G$, we call a matching $M\subset E(G)$ a rainbow matching if all its edges have distinct colors. An $(f,k)$-colored $r$-graph $G=(V,E)$ is an $r$-uniform multi-hypergraph whose edges are colored in $f$ colors such that every color class contains a $k$-matching. Denote by $f(r,k)$ the largest number $f$ of colors such that there exists an $(f,k)$-colored $r$-partite $r$-graph without a rainbow $k$-matching.  Recently, Glebov, Sudakov and Szab\'{o} \cite{glebov} gave an upper bound on $f(r,k)$.

\begin{thm}\label{le2}\cite{glebov}
For arbitrary integers $r, k\geq 2$, $f(r,k)<(r+1)^{r+1}(k-1)k^{2r}$.
\end{thm}

Now we consider the maximum number of copies of $K_s$ in a $kK_r$-free $r$-partite graph for $n_3\leq n_4\leq \cdots\leq n_r$.

\begin{proof}[Proof of Theorem \ref{main-3}]
Let $r\geq 4$, $n_1$, $n_2$ and $n_3$ be fixed integers. The proof is  by induction on  $(s,\sum_{i=4}^r (n_i-n_3))$.
 The base case of $s=2$ is verified for all $r$ and $n_1\leq n_2\leq \cdots\leq n_r$ by Theorem \ref{s2}. For every $s\geq 3$, the base case of $n_1\leq n_2\leq n_3=n_4=\cdots=n_r$ is verified for all $r$ by  Theorem \ref{main-2}. Assume that for all $r$, the theorem holds for all pairs $(s',\sum_{i=4}^r (n_i'-n_3'))$ such that $s'<s$ or $s'=s$ together with $\sum_{i=4}^r (n_i'-n_3')<\sum_{i=4}^r (n_i-n_3)$.

 Suppose $G\subseteq K_{n_1, \ldots, n_r}$ does not contain a copy of $kK_r$. Since $\sum_{i=4}^r (n_i-n_3)>0$, there exists an $i\in [4,r]$  such that $n_i> n_{i-1}$. Without loss of generality, assume that $i=r$. For $u\in V_r$, let $G(u)$ denote the $(r-1)$-partite graph on vertex classes $V_1, \ldots , V_{r-1}$, and a pair $\{v_i,v_j\}$ forms an edge in $G(u)$ if and only if $\{u,v_i\},\{u,v_j\}$ and $\{v_i,v_j\}$ are all edges in $G$.
 If there is a vertex $u\in V_r$ such that $G(u)$ is $kK_{r-1}$-free, then by induction on $s$, we have $k_s(u,G) = k_{s-1}(G(u)) \leq h^{(s-1)}_{k}(n_1,n_2,\ldots,n_{r-1})$. Moreover, by induction on $\sum_{i=4}^r (n_i-n_3)$, we obtain that $k_s(G\setminus\{u\})\leq h^{(s)}_{k}(n_1,n_2,\ldots,n_{r-1},n_r-1)$. Therefore,
 \begin{eqnarray*}
 k_s(G)&=&k_s(G\setminus\{u\})+k_s(u,G)\\
 &\leq&h^{(s)}_{k}(n_1,n_2,\ldots,n_{r-1},n_r-1)+h^{(s-1)}_{k}(n_1,n_2,\ldots,n_{r-1})\\
 &=&h^{(s)}_{k}(n_1,n_2,\ldots,n_{r-1},n_r).
 \end{eqnarray*}

 Otherwise, suppose that for all $u\in V_r$, there are at least $k$ vertex-disjoint copies of $K_{r-1}$ in $G(u)$. Since $G$ is $kK_r$-free, we have $k\geq 2$.
Let $H$ be an $(r-1)$-partite $(r-1)$-uniform multi-hypergraph on vertex classes $V_1,\ldots, V_{r-1}$. For any $u\in V_r$, if $\{u_1, \ldots, u_{r-1}\}$ forms  a copy of $K_{r-1}$ in $G(u)$, let $\{u_1, \ldots, u_{r-1}\}$ be an edge in $H$ with color $u$. Then  $H$  is $(n_r, k)$-colored.
Since $n_r\geq n_4\geq r^r(k-1)k^{2r-2}>f(r-1,k)$, by Theorem \ref{le2}, there is a rainbow $k$-matching $\{e_{i_1},\ldots, e_{i_k}\}$ in $H$. Thus, there are $k$ vertices $\{u_{i_1}, \ldots, u_{i_k}\}\subset V_r$ such that $\{e_{i_1}\cup \{u_{i_1}\}, \ldots, e_{i_k}\cup \{u_{i_k}\}\}$ forms a $kK_r$ in $G$, a contradiction.  Thus, we complete the proof.
\end{proof}

\vspace{10pt}\noindent
{\bf Acknowledgement.} The authors would like to thank two anonymous referees for their
helpful suggestions. The second author was supported by the National Natural Science
Foundation of China (No. 11701407) and Scientific and Technological Innovation Programs
of Higher Education Institutions in Shanxi (No. 183090222-S).

\begin{appendix}
\section{A proof of Theorem \ref{s1}.}

\begin{lem}\label{ma2}
For $n_1\leq n_2\leq n_3$ and $k\leq n_1$,
\begin{eqnarray*}
ex(K_{n_1,n_2, n_3}, kK_2)=(k-1)(n_2+n_3).
\end{eqnarray*}
\end{lem}

\begin{proof}
 First, we prove the lemma for $n_1=n_2=n_3=n$ by induction on $k$. Clearly, the lemma holds trivially for $k=1$. We assume that the result holds for all $k'$ with $k'< k\leq n$. Suppose $G$ is a $kK_2$-free 3-partite graph with vertex set $V=X\cup Y\cup Z$ and let
 $$ X=\{x_1, \ldots, x_{n}\},\quad Y=\{y_1, \ldots, y_{n}\}\quad\mbox{and}\quad Z=\{z_1, \ldots, z_{n}\}.$$
Define a partial order $\prec$ on $X\cup Y\cup Z$ such that
\[
x_1 \prec \cdots\prec x_{n},\ y_1 \prec \cdots\prec y_{n},\ z_1 \prec \cdots\prec z_{n},
\]
and vertices from different parts are incomparable. Assume that $G$ has maximal number of edges. Thus, $\nu(G)=k-1$. By Lemma 2.2, we may further assume that $G$ is stable. Let $T_0=\{x_1, y_1, z_1\}$ and $G'=G\setminus T_0$. Furthermore, let $\nu(G')=t$ and let $M'=\{e_1, \ldots, e_t\}$ be a largest matching in $G'$. If $G[T_0]$ is not a triangle, since $G$ is stable, there exist two vertex sets $V_i, V_j\in\{X, Y, Z\}$ such that $G[V_i\cup V_j]$ is empty. It follows that $G$ is a bipartite graph. Then by Lemma \ref{ma} with $r=2$, we conclude that $e(G)\leq 2(k-1)n$ and the lemma holds. If $G[T_0]$ is a triangle, then we have $k-4\leq \nu(G')\leq k-2$, where $\nu(G')\leq k-2$ follows from $G[T_0]$ being non-empty, and $\nu(G')\geq k-4$ follows from there being only three vertices in $T_0$ and from $\nu(G)=k-1$.
The proof splits into three cases according to the value of $\nu(G')$.

\textbf{Case 1.} $\nu(G')=k-2$. For every edge $\{u_i,v_i\} \in M'$, it is easy to see that the number of edges between $\{u_i,v_i\}$ and $T_0$ is at most 4 since $G$ is a 3-partite graph. Thus, there are at most $4(k-2)$ edge between $\cup_{e\in M'} e$ and $T_0$. If $|\Gamma(T_0)|>4(k-2)+3$, then we will find an edge between $T_0$ and $V(G')\setminus \left(\cup_{e\in M'} e\right)$. Without loss of generality, assume $\{x_1,u\}$ is such an edge. {Then} $M'\cup \{\{x_1,u\}, \{y_1,z_1\}\}$ forms a matching of size $k$, which contradicts the fact that $G$ is $kK_2$-free. If $|\Gamma(T_0)|\leq 4(k-2)+3$, then {by the induction hypothesis}, we have
\begin{align*}
e(G)&=|\Gamma(T_0)|+e(G')\\
&\leq 4(k-2)+3+2(k-2)(n-1)\\
&= 2(k-1)n - 2n + 2k - 1\\
&\leq 2(k-1)n.
\end{align*}

\textbf{Case 2.} $\nu(G')=k-3$. If $|\Gamma(T_0)|\leq 4n+2(k-3)$, then {by the induction hypothesis}, we have
\[
e(G)=|\Gamma(T_0)|+e(G')\leq 4n + 2(k-3)+2(k-3)(n-1)= 2(k-1)n.
\]
Thus, the lemma holds. If $|\Gamma(T_0)|> 4n+2(k-3)$, let $G''=G'\setminus (\cup_{e\in M'} e)$ {and consider} the edges between $T_0$ and $V(G'')$. Since there are at most $4(k-3)$ edges between $T_0$  and $(\cup_{e\in M'} e)$, the number of edges between $T_0$ and $V(G'')$ is at least $4n+2(k-3)+1-4(k-3)-3=4n-2k+4$. For any $u\in V(G)$ and $S\subset V(G)$, let $d(u,S)$ be the number of neighbors of $u$ in $S$. Then, it follows that
\[
d(x_1,V(G''))+d(y_1,V(G''))+d(z_1,V(G'')) \geq 4n-2k+4.
\]
Since $(Y\cup Z)\setminus (\cup M') \setminus T_0$ has at most $2(n-1)-(k-3)=2n-k+1$ vertices, we have $d(x_1,V(G''))\leq 2n-k+1$. Similarly, $d(y_1,V(G''))\leq 2n-k+1$ and $d(z_1,V(G''))\leq 2n-k+1$. Therefore,  for any $v\in\{x_1, y_1, z_1\}$,  $d(v,V(G''))\geq 4n-2k+4-2(2n-k+1)=2$. It follows from Hall's theorem that there exist three disjoint edges $\{x_1,u_1\}, \{y_1,u_2\}$ and $\{z_1,u_3\}$ with $u_1,u_2,u_3\in V(G'')$. These edges together with edges in $M'$ form a matching of size $k$, which contradicts the fact that $G$ is $kK_2$-free.

\textbf{Case 3.} $\nu(G')=k-4$.
Since $|\Gamma(T_0)|< 6n$, by the induction hypothesis,  we have
\[
e(G)=|\Gamma(T_0)|+e(G')\leq6n+2(k-4)(n-1)\leq 2(k-1)n.
\]
Thus, the lemma holds for $n_1=n_2=n_3=n$.

At last, we prove the lemma for the general case $n_1\leq n_2\leq n_3$ by induction on $n_2+n_3-2n_1$. Since we've already proven the base case $n_2+n_3-2n_1 = 0$, now assume that $n_2+n_3-2n_1>0$. There exists $i=2$ or 3 such that $n_i>n_{i-1}$. Without loss of generality, assume that $i=3$. If there exists $v\in Z$ such that $d(v)\leq k-1$, we have
\begin{align*}
e(G)&= d(v)+e(G\setminus v)\\
&\leq k-1+(k-1)(n_2+n_3-1)\\
&=(k-1)(n_2+n_3).
\end{align*}
If $d(v)\geq k$ for every $v\in Z$, since $|Z|\geq k$,  it is easy greedily to find a matching of size $k$, a contradiction. Thus, we complete the proof.
\end{proof}


\begin{proof}[Proof of Theorem 1.1]
The cases $r=2$ and $r=3$ follow from Lemmas \ref{ma} and \ref{ma2}, respectively.  Thus, {we are left with} the case $r\geq 4$ which we prove by induction on $k$. Clearly, the result holds for $k=1$. Assume that the result holds for all $k'< k$. Let $G\subseteq K_{n_1,\ldots,n_r}$ be a $kK_2$-free graph with the maximum number of edges. Thus, $\nu(G)=k-1$. Denote by $X_i$ the set of vertices in $V_i$ with degree at least $2k-1$ and put $x_i=|X_i|$ for $i=1, \ldots, r$. Let $n=n_1+\cdots n_r$ and $x=x_1+\cdots +x_r$. Now we divide the proof into two cases according to the value of $x$.

\textbf{Case 1.} $x\geq 1$. Let $X=\bigcup_{i=1}^r X_i$ and $G' =G\setminus X$.
 Since $d(u)\geq 2k-1$ for each  $u\in X$,  it is easy to see that $x\leq k-1$ and $\nu(G')\leq k-1-x$ because otherwise one could greedily find a matching of size $k$. Let $\bar{x}_i=x-x_i$ and $n_{i_0}-x_{i_0} =\min_{i\in [r]} \{n_i-x_i\}$. By the induction hypothesis, we have
\begin{eqnarray*}
e(G)&=& |\Gamma(X)|+e(G')\\
&\leq&\sum_{i<j}x_ix_j+\sum_{i=1}^rx_i\left(\sum_{j\neq i}(n_j-x_j)\right)+(k-1-x)\left(\sum_{i=1}^r(n_i-x_i)-\min_{i\in [r]}\{n_i-x_i\}\right)\\
&=&(k-1)n-(k-1)(x+n_{i_0}-x_{i_0})+\sum_{i<j}x_ix_j+\sum_{i=1}^rx_i(n_{i_0}-x_{i_0}-(n_i-x_i))\\
&\leq&(k-1)n-(k-1)(n_{i_0}+\overline{x}_{i_0})+\sum_{i<j}x_ix_j\\
&=&(k-1)(n-n_{i_0})-(k-1)\overline{x}_{i_0}+x_{i_0}\overline{x}_{i_0}+\sum_{i<j\atop i,j\neq i_0}x_ix_j\\
&\leq&(k-1)(n-n_1)-\overline{x}_{i_0}^2+\sum_{i<j\atop i,j\neq i_0}x_ix_j\\
&=&(k-1)(n-n_1)-\sum_{i\neq i_0}x_i^2-\sum_{i<j\atop i,j\neq i_0}x_ix_j\\
&\leq&(k-1)(n_2+\cdots+n_r),
\end{eqnarray*}
where the second inequality follows from $n_{i_0}-x_{i_0}-(n_i-x_i)\leq 0$ and the {third} inequality follows from $n_{i_0}\geq n_1$ and $x_{i_0}+\overline{x}_{i_0}=x\leq k-1$. Thus, the theorem holds.

\textbf{Case 2.}  $x=0$. Then all the vertices in $G$ have degree at most $2k-2$. Let $M=\{\{u_1,v_1\}, \ldots, \{u_{k-1},v_{k-1}\}\}$ be a largest matching of $G$, $A=\{u_1,\ldots, u_{k-1},v_1, \ldots, v_{k-1}\}$ and $B=V(G)\setminus A$. Since $M$ is a largest matching,  $B$ is an independent set of $G$. Let $t_i$ be the number of edges between $\{u_i,v_i\}$ and $B$. We claim that $t_i\leq 2k-2$. Otherwise, there exist $u,v\in B$ such that both $\{u_i,u\}$ and $\{v_i,v\}$ {are edges} of $G$, and then  $M'\setminus\{\{u_i,v_i\}\}\cup\{\{u_i,u\}, \{v_i,v\}\}$ forms a matching of size $k$,  a contradiction. Since $d(v)\leq 2k-2$ for every $v\in V(G)$ and $d_B(u_i)+d_B(v_i)=t_i$, we have $d_A(u_i)+d_A(v_i)\leq4k-4-t_i$. Thus,  we have
\begin{eqnarray*}
e(G)&=&e(A,B)+e(A) \\
&=&\sum_{v\in A}d_B(v) +\frac{1}{2}\sum_{v\in A}d_A(v)\\
&=&\sum_{i=1}^{k-1}(d_B(u_i)+d_B(v_i))+\frac{1}{2}\sum_{i=1}^{k-1}(d_A(u_i)+d_A(v_i))\\
&\leq &\sum_{i=1}^{k-1} t_i+\frac{1}{2}\sum_{i=1}^{k-1}(4k-4-t_i)\\
&= &\frac{1}{2}\sum_{i=1}^{k-1} t_i+\frac{1}{2}(k-1)(4k-4)\\
&\leq &(k-1)(3k-3)\\
&< &(k-1)(n_2+\cdots+n_r),
\end{eqnarray*}
where the last inequality follows from  $r\geq 4$ and $n_1\geq k$. Thus, we complete the proof.
\end{proof}
\end{appendix}

\end{document}